\documentclass{article}

\usepackage{amssymb}
\usepackage{mathtools}
\usepackage{mathdots}

\usepackage{amsmath}
\usepackage{amscd}
\usepackage{amsthm}
\usepackage{graphicx}
\usepackage[all]{xy}

\usepackage{color}          
\usepackage{epsfig}

\usepackage{lineno}

\newcommand{\C}{\mathbb {C}}

\newtheorem{theorem}{Theorem}
\newtheorem{lemma}[theorem]{Lemma}
\newtheorem{corollary}[theorem]{Corollary}

\newtheorem*{definition}{Definition}

\title{On hyperbolic cobordisms and Hurwitz classes of holomorphic coverings}
\author{Carlos Cabrera, Peter Makienko and Guillermo Sienra.}

\begin{document}

\maketitle
\footnotetext{This work was partially supported by PAPIIT IN102515 and CONACYT CB15/255633.}
 \begin{abstract}
In this article we show that for every collection $\mathcal{C}$ of an even number of polynomials, all of the 
same degree $d>2$ and in general position, there exist two hyperbolic $3$-orbifolds $M_1$ and $M_2$ with a M\"obius morphism $\alpha:M_1\rightarrow M_2$ such that the restriction of $\alpha$
to the boundaries  $\partial M_1$ and $\partial M_2$ forms a collection of maps $Q$ in the same conformal Hurwitz class 
of the initial collection $\mathcal{C}$. Also, we discuss the relationship between conformal 
Hurwitz classes of rational maps and classes of continuous isomorphisms of sandwich 
products  on the set of rational maps.  
\end{abstract}

\section{Introduction}

Cobordism theory has been studied widely since it was introduced by H. Poincar\'e 
in the context of  homology theory. Also R. Thom studied cobordism of embeddings.  Since then 
 there is an interest in cobordism of functions, for instance functions with stable singularities. 
Cobordims can be endowed with geometric structures such as symplectic structures, flat 
connections or complex structures.

For example, start with a pair of Kleinian groups $\Gamma_1$ and $\Gamma_2$ such that 
$\Gamma_1$ is a subgroup of finite index in $\Gamma_2$. The inclusion map generates a M\"obius morphism $\alpha:M(\Gamma_1)\rightarrow M(\Gamma_2)$ which is a finite degree orbifold covering. Since $\partial M(\Gamma_1)$ may be disconnected, the restriction 
$f:=\alpha|_{\partial M(\Gamma_1)}$ forms a collection of finite
degree holomorphic coverings from the components of $\partial M(\Gamma_1)$ to the components of $\partial M(\Gamma_2)$. 
In this situation, it is natural to say that the collection $f$ forms
a hyperbolic cobordism.

 With this point of view we avoid the homological 
language and will be interested in the following inverse problem: 

Given a collection $Q$ of  holomorphic finitely degree (orbifold) coverings, does there exists 
a pair of Kleinian groups  and a M\"obius morphism $\alpha$ which is conformally equivalent to 
the collection $Q$? 

Another motivation to this question is the  relational dictionary between rational maps and 
Kleinian groups. For this reason, the collection of maps will be
often  taken as a collection of rational endomorphisms of the Riemann sphere.  The main results of this 
article are Theorems \ref{TheoremA} and \ref{TheoremB}  below, these  are proven in Sections \ref{proofA}, 
\ref{ProofHurwitz} and \ref{proofTheoremB}. 

In the last section, we  characterize algebraically when two rational maps define the 
same conformal Hurwitz class. Also we briefly remark that the Hurwitz class of a 
rational map $R$ can be presented as a space of quasiconformal deformations of a 
semigroup of holomorphic correspondences and discuss the related questions. 

From now on our surfaces are compact surfaces with finitely many 
punctures which admit a hyperbolic (orbifold) structure of finite type. However, some 
of our results can be extended to the case of infinite type. 

We start with the following definition.
\begin{definition}
A branched covering of finite degree $d$ is a 
triplet $(R,S,S')$ where $S$ and $S'$ are finite collections of Riemann surfaces 
 and $R:S\rightarrow S'$ is a continuous surjective mapping so that 

\begin{enumerate}
 \item If $Y\subset S$ is a component then $Z=R(Y)\subset S'$ is 
 also a component and the restriction $R:Y\rightarrow Z$  is a degree $d_Y\leq d$ branched 
 covering map.  
\item There is a component $\tilde{Y}$ such that $d_{\tilde{Y}}=d.$
 \end{enumerate}
 We say that a branched  covering $R$ is \textbf{simple} whenever the number 
of components of $S$ coincides with the number of components of $S'.$ If $R$ is 
simple and $S$ is connected then we say that $R$ is a 
 \textbf{single} branched covering.
\end{definition}

Branched coverings between Riemann surfaces have been studied widely 
in the literature. We are interested in the following basic examples. 

\begin{enumerate}
 \item Rational maps, these are single branched self coverings of the sphere.
 \item Let $\Gamma$ be a Kleinian group and $G<\Gamma$ be a subgroup with 
$\Omega(G)=\Omega(\Gamma)$ then the natural projection $R:
S(G)\rightarrow 
S(\Gamma)$ is a holomorphic branched covering map. When $\Omega(G)\neq \Omega(\Gamma)$, in general, the inclusion map does not induces a holomorphic covering. Here $\Omega(\Gamma)$ denotes the discontinuity set of $\Gamma$.

It is known that the equality $\Omega(G)=\Omega(\Gamma)$ holds whenever $G$ is either normal  or a 
subgroup of finite index. However, there are examples of 
groups $G$ with $\Omega(G)=\Omega(\Gamma)$ but such that $G$ is not normal and has infinite index in 
$\Gamma.$ If $G$ and $\Gamma$  are finitely generated and the limit set of $\Gamma$ 
is not a subset of a round circle, then by the Ahlfors finiteness theorem $G$  necessarily has finite index in $\Gamma.$

\end{enumerate}

Any branched covering can be regarded as a collection of single coverings, we call
each of them a \textit{single component} of the branched covering. We say that a branched 
covering $R$ is a holomorphic covering whenever every single component is a 
holomorphic (orbifold) unbranched covering between hyperbolic surfaces (orbifolds).

Given a holomorphic covering $(R,S,S')$, we can improve  $(R,S,S')$  into a 
simple covering in the following way: if $X$ and $Y$ are components of $S$ so that 
$R(X)=R(Y)=Z\subset S'$ then consider a conformal copy $Z'$ of $Z$. 
Let $R':X\rightarrow Z'$ be the respective holomorphic covering, now the holomorphic 
covering $X\sqcup Y\rightarrow R'(X)\sqcup R(Y)$ is simple. By induction on the number of components we construct 
a simple holomorphic covering  $(Q,T,T')$ such that for every 
single component $R:X\rightarrow Z$ of $(R,S,S')$ there exists a single component 
$Q:X'\rightarrow Z'$ of $(Q,T,T')$
 and two conformal homeomorphisms $\phi:Z\rightarrow Z'$ and $\psi:X\rightarrow 
 X'$ so that $\phi\circ R=Q\circ \psi.$ The previous discussion also motivates the following definition of Hurwitz classes for non-connected branched coverings. 
  
\begin{definition}
Let $f:S \rightarrow S'$ be a branched covering,  the {\bf Hurwitz class} 
$H(f)$ of $f$ consists of the triples $(g,N,N')$ so that $g:N\rightarrow N'$ 
is a branched covering  and there exist orientation preserving homeomorphisms 
${\varphi}:S\rightarrow N$ and ${\psi}: S' \rightarrow N'$ such  that ${\varphi} \circ 
f=g \circ {\psi}$. 
\end{definition}
If $f:S\rightarrow S'$ is a single branched covering then $H(f)$ coincides with the 
classical Hurwitz space of $f$. If $f$ is a simple branched covering, then 
$$H(f)=\bigotimes_Y H(f, Y, f(Y))$$ where the product is taken over the 
connected components $Y$ of $S.$ 

Given a holomorphic map $f$, the set $$CH(f)=\{(g,N,N')\in H(f), \phi, \psi \textnormal{ 
conformal}\}$$ is called the \textit{conformal Hurwitz class} of the holomorphic covering $f.$

For example, if $f:\overline{\C}\rightarrow \overline{\C}$ is a 
finite degree branched covering then $H(f)$ contains a rational map of the same 
degree. When $f$ is a rational map the set $H(f)\cap Rat(\C)=S(f)$ is known as the 
\textit{Speisser class} of $f$ and was introduced into holomorphic dynamics by A. Eremenko and 
M. Lyubich in \cite{EreLyub}. By Teichm\"uller's Theorem,  if $f$ is  holomorphic of finite degree and if $g\in  H(f)$ is holomorphic, then  the associated homeomorphisms $\phi$ and 
$\psi$ can be taken  quasiconformal. For general holomorphic maps, this is not true. It is not even clear whether $\phi$ and $\psi$ can be taken to be local quasiconformal maps for an 
infinite degree holomorphic map $f.$ A plausible counterexample is an entire map $f$ with
positive Lebesgue set of singularities of $f^{-1}.$

The \textit{hyperbolic cobordism} between two holomorphic coverings $(R_1,S_1,S'_1)$ and 
$(R_2,S_2,S'_2)$ is given by the triple $(\Re,M,M')$ satisfying the following conditions:
\begin{enumerate}
\item There are geometrically finite Kleinian groups $\Gamma,\Gamma'< PSL(2,\C)$ such 
that 
$$M=M(\Gamma)=(B\bigcup{\Omega}({\Gamma}))/{\Gamma}$$ and $$M'=M(\Gamma')=(B\bigcup{\Omega}
(\Gamma)')/{\Gamma}'.$$  Hence, $M$ and $M'$ are 
oriented hyperbolic $3$-orbifolds with natural projections ${\pi}:B\bigcup{\Omega}
({\Gamma}) \rightarrow M$, 
and ${\pi}':B\bigcup{\Omega}({\Gamma}') \rightarrow M'$. The map ${\Re}:M \rightarrow 
M'$ is a \textit{surjective M\"obius morphism}, that is, there exists an orientation preserving M\"obius 
map $\alpha$ such that the following diagram commutes: 

\begin{displaymath}
    \xymatrix{B\bigcup{\Omega}({\Gamma})
         \ar[d]_{\pi} \ar[r]^{\alpha} & 
\ar[d]^{\pi'}B\bigcup{\Omega}({\Gamma}')\\M
           \ar[r]^{\Re}  & M'}\tag{1}\label{diag.basic}
\end{displaymath}

 \item The boundary ${\partial}M$ is conformally equivalent to $\bigsqcup S_{i}$ and 
 ${\partial}M'$ conformally equivalent to $\bigsqcup S'_{i}$, so that 
 $$(\partial M,\partial M',\Re|_{\partial M})\in \bigotimes_{i=1}^{2} CH(S_i,S'_i,R_i).
 $$
 \end{enumerate}

 Hence $\Re$ is a local isometry  between the respective hyperbolic metrics on 
 $M(\Gamma)$ and $M(\Gamma')$ induced by the Kleinian groups $\Gamma$ and 
 $\Gamma'$, respectively.

Given two holomorphic coverings $R_1$ and $R_2$, if there exists a hyperbolic 
cobordism between $R_1$ and $R_2$, we will say that $R_1\sqcup R_2$ 
forms a \textit{cobordant family of holomorphic coverings}, or that $R_1$ is 
\textit{hyperbolically cobordant} to $R_2.$

Given a finite degree holomorphic branched covering $R:M \rightarrow N$, between 
Riemann surfaces  $M$ and $N$, there are many ways to transform $R$ into a holomorphic 
covering between hyperbolic orbifold structures supported on $M$ and $N$. We consider the 
simplest construction depending on the ramification data of $R$ and a finite subset 
$A\subset N$ as follows: first restrict $R$ to $R:\{S=M\setminus R^{-1}(A)\}\rightarrow 
\{S'=N\setminus A\}$. Second, using the ramification data of $R$  produce  
orbifold structures on $S$ and $S'$ so that $R$ is a holomorphic (orbifold) covering 
between hyperbolic orbifold structures supported on $S$ and $S'$, respectively.
In particular, if $A=\emptyset$, then the canonical orbifold structure on $M$ and $N$ 
defined by the ramification data of $R$ must be hyperbolic. For instance, in the case 
where $R(z)=z^n$, the set $A$ must be non-empty and $card(A\setminus \{0,\infty\})\geq 
1$.

If $A=V(R)$ is the set of critical values of $R$, and the surfaces $S$ and $S'$ are hyperbolic then 
the triple $(R,S,S')$ is called the \textit{canonical holomorphic representative} of the 
holomorphic branched covering $R:M\rightarrow N.$

\medskip

\textbf{Examples}.
\begin{enumerate}

\item The {\bf null cobordism}  where $S$ and 
$S'$ are connected is related to the extension of a single holomorphic covering to the 
respective 3-hyperbolic spaces. This situation has been studied in \cite{CabMakSie} with applications 
to holomorphic dynamical systems. In particular, in \cite{CabMakSie} the authors gave the construction of a 
geometric extension for generic rational maps. The present article develops 
the geometrical part of \cite{CabMakSie} in the case of a collection of holomorphic 
coverings.

\item The trivial cobordisms. Consider the identity maps $Id_i:S_i\rightarrow 
S_i$, where $S_i$ is a Riemann surface for $i=1,2.$ Then the existence of a 
cobordism between $Id_1$ and $Id_2$  reduces to the existence of a hyperbolic 
manifold with boundary conformally equivalent to $S_1 \sqcup S_2$. Then we have:

\begin{itemize}
 \item If $S_1$ is anticonformally equivalent to $S_2$, then by the Bers's simultaneous 
 uniformization theorem there exists a Fuchsian group uniformizing the surfaces 
 $S_1$ and $S_2$, so that $S_1\sqcup S_2$ is conformally equivalent to the boundary 
 of a hyperbolic 3-manifold.
 
 \item For any surface $S_1$ consisting of a finite number of hyperbolic components, 
 there exists a connected hyperbolic surface $S_2$  such that $S_1\sqcup S_2$ can be 
 uniformized by a functional geometrically finite group. This uniformization is given by the Klein-Maskit 
 combination theorems in such a way that $S_1\sqcup S_2$ bounds an oriented hyperbolic $3$-orbifold.  This 
 observation will be needed in the proof of Theorem \ref{TheoremA}.

\item Whenever $S$ is a compact hyperbolic closed connected  Riemann surface 
with an even number of cusps, there exists a Schottky type group uniformizing $S$ so 
that $S$ is conformally equivalent to the boundary of a $3$-hyperbolic manifold.

\end{itemize}

\end{enumerate}

\textbf{Connected transitivity}. Given three single holomorphic coverings $R_{1}$, 
$R_{2}$, and $R_{3}$; such that the pairs   $(R_{1},R_{2})$ and  $(R_{2},R_{3})$ 
are each hyperbolically cobordant by manifolds $M_1$ and $M_2$. Assume that the canonical homorphisms 
$\pi_1(S_i)\rightarrow \pi_1(M_i)$ and $\pi_1(S'_i)\rightarrow \pi_1(M'_i)$ are injective, then  
$R_{1}$ is hyperbolically cobordant to $R_{3}$. In fact, this follows from the 
Thurston's hyperbolization theorem.
 
In general, without the single and injective assumptions, is not clear that the manifold, resulting by gluing $M_1$ and $M_2$ along the boundary components associated to $R_2$, is hyperbolic. This is because the result of gluing hyperbolic manifolds along the boundary may not be hyperbolic. For instance, consider 
a geometrically finite hyperbolic 3-manifold $M$ which has an essentially 
embedded annulus $A.$ Let $S_1$ and $S_2$ be not necessarily different
components of $\partial M$ containing the boundary of $A$.
Take a copy of $M$, say $M'$, and make $V=M\sqcup_{S_1\sqcup S_2} M'$  by 
gluing $M$ and $M'$ along $S_1\sqcup S_2$. Then $V$ does not 
accept a hyperbolic metric since $V$ contains a torus which is not homotopic to the ideal boundary of 
$V.$ This indicates that there might be obstacles  to the existence 
of a hyperbolic cobordism between multiple  coverings. 

Now we formulate our first main 
theorem.

\begin{theorem}\label{TheoremA}
Given a simple holomorphic covering $F_1$, there 
exists another holomorphic covering $F_2$ such that  $F_1\sqcup F_2$ forms a family 
of cobordant holomorphic coverings.
\end{theorem}

Moreover, if $F_1$ has  a single component $R_0$ with degree $deg(R_0)>1$,
then the covering $F_2$ contains only one single component, 
say $Q_0$,  with degree bigger than $1$ and $deg(Q_0)=\sum_{i=0} ^{n} deg 
(R_i)-n$ where $R_i$ are single components of $F_1$ and $n+1$ is the number of these components. 

We need the following definition.

\begin{definition}
A holomorphic covering $Q:S\rightarrow S'$ is called an anticonformal copy of the 
holomorphic covering map $R:T\rightarrow W$ if there are anticonformal 
 homeomorphisms $\alpha:S\rightarrow T$ and $\beta:S'\rightarrow W$ so that 
 $\beta\circ Q=R\circ \alpha.$ Given a holomorphic covering $R$, we call the Hurwitz 
class $H(R)$ \textbf{symmetric} if and only if $H(R)$ contains an anticonformal copy of an 
 element $g\in H(R)$. Finally, we say that a holomorphic covering is symmetric if its 
 Hurwitz class is symmetric. 
\end{definition}

Let us note that if $f:\overline{\C}\rightarrow \overline{\C}$ is  a branched covering, then $H(f)$ is  
symmetric whenever $H(f)$ contains a real rational map, that is, all coefficients are 
real. In particular, if $B$ is a Blaschke endomorphism, then $H(B)$ is symmetric. Moreover, if 
$ \bigsqcup_B H(B)$ is the union of all Hurwitz classes of Blaschke endomorphisms and 
$g\in \bigsqcup_B H(B)$ then $H(g)$ is symmetric. Also for every natural number $d$, by 
the Theorem 3.4 in \cite{Edmondscover}, the set $ \bigsqcup_B H(B)\cap 
Rat_d(\C)$ is connected  and contains an open and everywhere 
dense subset of the space of $Rat_d(\C).$ Here 
$Rat_d(\C)$ denotes the set of rational maps of degree $d$. A Blaschke endomorphism is a rational map $B$ with $B^{-1}(\mathbb{D})=\mathbb{D}$ where $\mathbb{D}$ is the open unit disk in $\mathbb{C}$. 

In general it is not clear that the Hurwitz class of any finite degree branched covering between 
closed Riemann surfaces is symmetric. But we believe that is true for Hurwitz classes of 
rational maps.

\begin{definition}
Two cobordant holomorphic coverings $R_1$ and $R_2$ are  called \textbf{simply 
cobordant} if and only if $M$ and $M'$ are homeomorphic to $S_1\times [0,1]$ and 
$S'_1\times [0,1]$.
\end{definition}

\begin{theorem}\label{Th.Hurwitz}
Two symmetric single holomorphic coverings $R_1$ and $R_2$ belong to the same Hurwitz 
class if and only if $R_1$ and $R_2$ are simply cobordant. 
\end{theorem}

Theorem \ref{TheoremA} shows that any finite family of coverings can be 
included in a family of cobordant coverings which is non simple and includes 
single univalent components on the boundary. Theorem \ref{Th.Hurwitz} gives 
a condition for when a pair of holomorphic coverings of the Riemann sphere is 
cobordant. 

With Theorem \ref{Th.Hurwitz} at hand we improve  Theorem \ref{TheoremA} into Theorem \ref{TheoremB}. First recall that a holomorphic polynomial  map $P:
\C\rightarrow \C$ of  degree $d>1$  is in general position if there are $d-1$ different finite critical values $V(P)$. It is known that two polynomial maps in general position belong to the same 
Hurwitz class if and only if these polynomials have the same degree. Also, 
every polynomial in general position is symmetric. 

\begin{theorem}\label{TheoremB} The canonical holomorphic 
representatives of every collection of an even number of polynomials in general position of 
the same degree $d>2$ form a hyperbolic cobordism.
\end{theorem}

\par\medskip\noindent \textbf{Acknowledgments.}
The authors would like to thank M. Kapovich for pointing out the ideas in Brooks's deformation 
theorem which simplified a previous version of the proof of Theorem \ref{TheoremB}.

\section{Some background on Kleinian groups}

For the convenience of the reader here we collect some facts from Kleinian group theory which will be used in this article. We follow the books of M. Kapovich \cite{KapovichHyp} and A. Marden \cite{MardenBook} which give a modern introduction to Kleinian groups.

Denote by $B$ the 
Poincar\'e model of the hyperbolic $3$-space, that is, the unit ball in $\mathbb{R}^3$ 
equipped with the Poincar\'e metric. Given a group $\Gamma$ of automorphisms of the Riemann sphere. We denote by $\Omega(\Gamma)$ the discontinuity set of $\Gamma$ on $\overline{\C}$.  The isometry group of $B$ acts on the Riemann sphere $\overline{\C}=\partial B$ as the whole group of M\"obius transformations $Mob(\C)$ including anticonformal automorphisms.  A discrete subgroup  $\Gamma$ of $ Isom(B)$
is a Kleinian group if $\Omega(\Gamma)\neq \emptyset$. 
Historically, a Kleinian group is defined as a subgroup of orientation preserving isometries of $B$, but we need the extended definition in order to apply Brook's deformation theorem.  Also we follow the definition in \cite{KapovichHyp} where it is shown that many classical theorems for orientation preserving Kleinian groups extend to the general case without many difficulties. 

Define $S(\Gamma)=\Omega(\Gamma)/\Gamma$ 
and $M(\Gamma)=(B\sqcup \Omega(\Gamma))/\Gamma$ and note that $S(\Gamma)=\partial M(\Gamma)$.

Both spaces $S(\Gamma)$ and $M(\Gamma)$ can be endowed with a hyperbolic orbifold structure.  For an orbifold $O$, let $|O|$ be the underlying space of $O$. When $\Gamma$ contains orientation reversing elements, one has to be cautious with the fact that  $|S(\Gamma)|$ is a proper subset of $\partial |M(\Gamma)|$. The points in 
 $\partial |M(\Gamma)| \setminus |S(\Gamma)|$ are interior points in the orbifold structure contained in the singular locus. In other words, neighborhoods of  these points are modeled by the quotient of a ball by the action of a finite group of isometries of $B$. The simplest example to have in mind is the space $X$ which is the quotient of $\C$ by the map $z\mapsto \overline{z}$. Then $X$ admits the structure of a manifold with 
boundary homeomorphic to the closure of the upper half-space.  Alternatively, $X$ also possess the structure of an orbifold without boundary where the real line is the singular locus of the orbifold.

\begin{definition}
A 3-manifold $M$ is called geometrically finite if there exists a compact submanifold with boundary $M_0$ such that $M\setminus M_0$ is a disjoint union of finitely many pieces $V_i$ satisfying
either that
\begin{itemize}
 \item $V_i$ is homeomorphic to $\mathbb{S}^1\times (\mathbb{D}\setminus \{0\})$, where $\mathbb{D}$ is the open unit disk in $\C$ and $\mathbb{S}^1$ is the unit circle in $\C$; or
 \item $V_i$ is homeomorphic to $[0,1] \times (\mathbb{D}\setminus \{0\})$ so that the 
 punctured disks $\{0\}\times (\mathbb{D}\setminus \{0\})$ and $\{1\}\times (\mathbb{D}\setminus \{0\})$ belong to $\partial M.$ 
 \end{itemize}
 
 A Kleinian group $\Gamma$ is called geometrically finite if and only if contains a finite index subgroup $\Gamma_0$ such that $M(\Gamma_0)$  is geometrically finite. 
\end{definition}

The pieces $V_i$ in the definition are usually known as solid cusp torii and solid pairing tubes respectively.
There are many equivalent definitions of geometrically finite Kleinian groups, see for example \cite{KapovichHyp} and \cite{MardenBook}.

\subsection{Pinching}\label{subsectionpinching} In what follows we describe the pinching procedure for a finite family of disjoint simple closed geodesics on a Riemann surface $S$ contained in $S(\Gamma)$ for a geometrically finite Fuchsian group $\Gamma< PSL(2,\C)$, according to theorems of B. Maskit \cite{MaskitKleinII} and K. Ohshika \cite{OhshikaGeoFin}, see also the pinching theorem of Section 5.15 in  \cite{MardenBook}.

For $r>0$ let $A_r=\{z\frac{1}{r}<z<r\}$ be the round symmetric annulus and define $F(z)=z|z|$ in $A_r$, note that $F(z)$ is quasiconformal. Take the sequence $\mu_n$ of Beltrami differentials on $A_r$  defined by $\mu_n=\frac{\overline{\partial}F^n}{\partial F^n}$ 
where $F^n$ is the $n$-th iterated of $F$ and the partial derivatives are taken in the sense of distributions. Then $\|\mu_n\|\rightarrow 1$ as $n\rightarrow \infty.$

Let  $l_i$ be a finite collection of disjoint simple closed 
geodesics in $S$, then by the collar lemma there exists $r_0$ and a family of conformal embeddings  $h_i:A_{r_0}\rightarrow S(\Gamma)$ with $h_i(\mathbb{S}^1)=l_i$ and the closed sets $\overline{h(A_{r_0})}$ are mutually disjoint. By taking the simultaneous push-forward of $\mu_n$ by the maps $h_i$, we obtain a sequence $\tilde{\nu}_n$ of Beltrami differentials on $S$ supported on the union of the annular neighborhoods $h_i(A_{r_0})$. Now lift the sequence $\tilde{\nu}_n$ over $\Omega(\Gamma)$ by the natural projection $\Omega(\Gamma)\rightarrow S(\Gamma)$ to get a sequence  $\nu_n$ of Beltrami differentials in 
$\Omega(\Gamma)$ with $\|\nu_n\|\rightarrow 1.$ If $f_n$ is a solution of the Beltrami equation with coefficient $\nu_n$ then the
group $\Gamma_n=f_n\circ \Gamma \circ f_n^{-1}$ in $PSL(2,\C)$  is quasifucshian. In case that $\Gamma$ acts on $\mathbb{D}$ and  $S=\mathbb{D}/\Gamma$ then all the maps $f_n$ are holomorphic outside $\overline{\mathbb{D}}.$

 Then the following theorem is true.
 
 \begin{theorem}\label{th.pinching}
  Let $\Gamma_n$ be a family of quasifuchsian groups 
  as constructed above. After taken a suitable subsequence there exists a geometrically finite  Kleinian group $\Gamma_\infty=\lim \Gamma_{n_k}$  in the topology of convergence on generators so that  
  
  \begin{itemize}
   \item $\Gamma_\infty\simeq \Gamma.$
   \item The interior of $M(\Gamma_\infty)$ is homeomorphic to the interior of $M(\Gamma).$
   \item The surface $S(\Gamma_\infty)$ is homeomorphic to $S(\Gamma)\setminus (\bigcup l_i).$ Even more, the homeomorphism can be chosen to be holomorphic outside $\bigcup h_i(A_{r_0})$ and each $l_i$ determines a pair of punctures in $S(\Gamma_\infty),$
   
    \end{itemize}
 \end{theorem}

 The previous theorem is proved by Maskit  for functional groups, with the condition that the closed simple curves $l \in S(\Gamma)$  to be pinched must
 have loxodromic representatives in the group, which represent different conjugacy classes. Last condition  always holds for simple closed geodesics which belong to the same connected component of $S(\Gamma)$ for a given quasifuchsian group $\Gamma$. In \cite{OhshikaGeoFin},  Ohshika extends the theorem of Maskit to all geometrically finite groups. Our version follows the exposition of Marden in the pinching theorem of Section 5.15 of \cite{MardenBook}.\\

\textbf{Klein-Maskit combination theorem.} We will need the following theorem in the proof of Theorem \ref{TheoremA}. See \cite{MaskitKleinII}.

\begin{theorem}\label{th.KleinMaskit}[Klein-Maskit's combination Theorem I 
\cite{MaskitBook}.] For i=1,2, let ${\Gamma}_{i}$ be a Kleinian group with region of 
discontinuity ${\Omega}_{i}$  and a fundamental region $F_{i}$. Assume that there is a simple closed loop ${\gamma}$ contained in the interior of $F_{1}\bigcap F_{2}$, 
bounding two complemented discs $D_1$ and $D_2$ with $\overline{D}_i\subset F_{i}$. 
Then ${\Gamma}=\langle {\Gamma}_{1}, 
{\Gamma}_{2}\rangle$ is a Kleinian group, such that:
\begin{enumerate}
\item The group $\Gamma$ is isomorphic to the free product $\Gamma_1\ast 
\Gamma_2$. If ${\Gamma}_{1}$ and ${\Gamma}_{2}$ are 
geometrically finite, then ${\Gamma}$ is so. 

\item Let $S_i=K_i/Stab(K_i)$ be surfaces where $K_i\subset \Omega(\Gamma_i)$ are 
the components containing $\gamma$ and $Stab(K_i)< \Gamma_i$ are their 
respective stabilizers. Then $S(\Gamma)$ is homeomorphic to
$$(S(\Gamma_1) \setminus S_1) \sqcup 
(S(\Gamma_2)\setminus S_2)\sqcup (S_1\#S_2)$$ where $S_1\#S_2$ is the connected sum of the surfaces $S_1$ and $S_2$ along the respective projections of 
$D_1$ and $D_2.$

Even more, this homeomorphism can be chosen holomorphic on $S(\Gamma)\setminus S_1 \# S_2$.
\item The manifold  $M(\Gamma)$ is homeomorphic to the disk sum $M(\Gamma_1)$ with $M(\Gamma_2)$ induced by  the disks determining the 
connected sum $S_1\# S_2.$
\end{enumerate}
\end{theorem}

\textbf{Disks patterns and Brook's deformation theorem.} The following construction is needed in the proof  of Theorem \ref{TheoremB}.

 \begin{definition} Let $\Gamma$ be a geometrically finite
 torsion free Kleinian group. Then a collection $K$ of closed sets $K_i\subset S(\Gamma)$
 is called a \textbf{round disk collection} if and only if the set $K$ consists of finitely many elements 
 and every element $K_i$  is either a homeomorphic projection of a compact round disk 
  $D\subset \Omega(\Gamma)$ to $S(\Gamma)$ or is a closed punctured disk in $S(\Gamma)$ where the puncture corresponds to a cusp of $S(\Gamma)$ and $K_i$ is covered by a round disk 
  $D\setminus \{p\}\subset \Omega(\Gamma)$, where $D$ is precisely
  invariant under its parabolic stabilizer $\gamma \in \Gamma$ and $p$ is the fixed
  point of $\gamma$ with $p\in \partial D.$
 \end{definition}
 
\begin{definition}
 A finite round disk collection $K\subset S(\Gamma)$ is called a \textbf{pattern of round disks} if and 
 only if the following holds:
 \begin{enumerate}
  \item No point in $S(\Gamma)$ is covered by the interior of more than two disks in $K.$
  \item Given two different disks $K_i$ and $K_j$, then either the interiors of $K_i$ and $K_j$ are disjoint or 
  their boundaries are orthogonal.
 \end{enumerate}
\end{definition}

 Each disk $K_i\subset K$ is covered by a round disk $D_i\in \overline{\C}$.  If $C(D_i)$
 is the convex hull of $\partial D_i$ in $B$ with respect to the Poincar\'e metric, then $C(D_i)$ is 
 invariant under the stabilizer of $D_i$ in $\Gamma.$ Let $V(D_i)$ be the component of $\overline{B}\setminus C(D_i)$ containing $D_i$. Let $Y(K_i)\subset M(\Gamma)$ be the projection of $V(D_i)$ in $M(\Gamma)$ and $$M_K=M(\Gamma)\setminus \bigcup_{K_i\in K} 
 Y(K_i),$$ then on the manifold $M_K$ there exists a 
 natural polyhedral geometric structure $\mathcal{G}$ which on the interior of  
 $M_K$ coincides with the hyperbolic structure of $int(M(\Gamma))$, on $M(\Gamma)\cap (\bigcup\partial Y(K_i))$ the structure 
 $\mathcal{G}$ is a polyhedral piecewise 
 geodesic structure, on the remaining boundary components of $M_K$ the structure coincides with the 
 M\"obius structure inherited from $M(\Gamma)$. 
 
 Let $\tilde{K}$ be the collection of all round disks in $\Omega (\Gamma)$ which 
 cover all $K_i\subset K$. Let $\Gamma_K<Isom(B)$ be the group 
 generated by $\Gamma$ and the reflections with respect to the circles $\partial D$
 for $D\in \tilde{K}. $ Then Theorem 13.1 in \cite{KapovichHyp} states as follows.
 
 \begin{theorem}\label{th.Theorem6} The group $\Gamma_K$ is geometrically finite and 
 $M(\Gamma_K)$ is an orbifold  diffeomorphic to $M_K$ equipped with the structure
$\mathcal{G}$.
\end{theorem}
 
 \noindent \textbf{Remarks}: 
 \begin{enumerate}
 \item In particular, if a component $\Omega_0\subset \Omega(\Gamma)$ does not 
 intersect $\tilde{K}$ then the stabilizer of $\Omega_0$ in $\Gamma_K$ coincides with 
 stabilizer of $\Omega_0$ in $\Gamma.$ Therefore, if a disk pattern $K$ completely covers exactly one component $S\subset \partial M(\Gamma)$, then $\partial(M(\Gamma_K))$ is a $2$-dimensional orbifold which is conformally equivalent to $\partial M(\Gamma)\setminus S.$
 
 \item The charts around points on 
 $\partial(Y(K_i))$ are modeled 
 with the quotient of the unit three dimensional ball by a finite group, this group is generated 
 by reflections on planes passing through the origin. In particular, in this structure the points in $\partial Y(K_i) 
 \setminus K_i$ are interior points of $M_K$ equipped with the structure $\mathcal{G}.$

 \end{enumerate}

 The following simple example shows how this procedure works. Let $S$ be any Riemann surface and $\Gamma$ be a Fuchsian group uniformizing $S.$ Then $M(\Gamma)$ is a hyperbolic manifold homeomorphic to $S\times [0,1]$, the boundary of $M(\Gamma)$  consists of $S$ and an anticonformal copy of $S$. Let $\tau$ be the reflection with respect to the unit circle, then $\tau$ commutes with $\Gamma$. Let $G=\langle \Gamma, \tau \rangle$. Thus $G$ is a Kleinian group and $M(G)$ is a non orientable 
 hyperbolic orbifold, so that $\partial M(G)$ is conformally equivalent to $S.$
 The underlying space of the orbifold $M(G)$ is 
 a manifold which still is homeomorphic to $S\times [0,1]$ but now, one of the components consists of interior points of the orbifold structure on  $M(G).$
 
The following theorem justifies the existence of pattern of disks for a quasiconformal deformation of a given geometrically finite group. This theorem is part of the proof of Brooks's orbifold 
 deformation theorem. More precisely, see the steps 1 to 4 in the proof presented 
 in Section 13.5 of \cite{KapovichHyp}.  The statement is as follows.
 
 \begin{theorem}\label{th.geofintor}
For any torsion free geometrically finite Kleinian group $\Gamma$ there exists a quasiconformal homeomorphism 
 $h$ such that the group $\Gamma_h=h\circ \Gamma\circ h^{-1}$ is so that 
 $S(\Gamma_h)$ admits a pattern $K$ which completely covers   $S(\Gamma_h).$
 \end{theorem}
 
 With this background we can proceed to prove our theorems.

\section{Proof of Theorem \ref{TheoremA}}\label{proofA}
 
 \subsection{Connected sums of single coverings}\label{subconnectedsums}

Now we reproduce a topological operation between  branched coverings 
which is a sort of ``connected sum" of coverings.   This 
operation consists of taking the connected  sum of the target surfaces and a pull-back with respect to the branched coverings. Note that there are different ways to make a pull-back. We choose the simplest as follows.

Start with two single finite degree branched coverings $R_1:S\rightarrow S'$ and 
$R_2:T\rightarrow T'.$ We  construct $R_0=R_1\# R_2$,  the connected sum 
of  branched covering maps $R_1$ and $R_2$, as  a branched covering  
between two surfaces $U$ and $W$ such that $deg(R_0)=deg(R_1)+deg(R_2)-1$ and
$W=T'\# S'$ is the connected sum with respect to topological disks $D_S\subset S'$ and
$D_T\subset T'$ not containing the branched points (critical values) of $R_1$ and $R_2$, respectively.\\ 

{\bf Topological construction of $U$}. 
Set $deg(R_1)=n$ and $deg(R_2)=m$. Let $S''=S'\setminus D_S$ and 
$T''=T'\setminus D_T$ and $h:\partial D_S \rightarrow \partial D_T$ be a gluing 
homeomorphism. Let  $S_0=R^{-1}_1(S'')\subset S$ and 
$T_0=R_2^{-1}(T'')\subset T.$
Let us fix the following system of curves and homeomorphism.
\begin{enumerate}
\item Take two components of the boundaries, $\alpha_0$
a component of $\partial S_0$ and  $\beta_0$ a component of $\partial T_0.$ Let  $\phi_0:\alpha_0 \rightarrow \beta_0$ be a homeomorphism such that 
$R_2\circ \phi_0=h\circ R_1$.

 \item If $\gamma$ is another component of either $\partial S_0$ or $\partial T_0$, and different from 
 $\alpha_0$ and $\beta_0$, then fix a homeomorphism $\phi_\gamma$ which is either $h\circ R_1|
 _\gamma$ or $h^{-1}\circ R_2|_\gamma$, depending on the case. 
 For $i=1,...,m-1$, let $\{S_i\}$ be $m-1$ copies of $S''$ and for $j=1,...,n-1$ let 
 $\{T_i\}$ be a family of $n-1$ copies of $T''$. Then $$U=S_0\sqcup T_0\sqcup 
 \{\sqcup S_i\}\sqcup \{\sqcup T_i\}/\sim$$ where the quotient is taken according to the 
 system of homeomorphisms. 
 More precisely, the homeomorphism $\phi_0$ identifies $\alpha_0$ with $\beta_0$, and 
 the map $\phi_\gamma$  identifies  the component $\gamma$, which is either in $\partial S_0$ or in $\partial T_0$, with the respective copies of $T''$ or $S''$. The identification is taken in such a way that 
 $U$ is a connected surface and there exists a branched covering $R_1\#R_2:U\rightarrow 
 W$ so that $R_1\# R_2|_{S_0}=R_1:S_0\rightarrow S''$ and 
 $R_1\# R_2|_{T_0}=R_2:T_0\rightarrow T''.$ The restriction of $R_1\# R_2$ on each one of the remaining copies, of either $T''$ or  $S''$,  is univalent.
\end{enumerate} 
 
 Then, we have that $genus(U)=m \times genus(S)+n \times genus(T)$. Moreover,
 punctures and holes satisfy the same equation as the genus.

If $R_1$ and $R_2$ are holomorphic branched coverings between Riemann surfaces,
then by taking  a conformal gluing in the  construction of $U$ we can assume that 
$R_1\#R_2:U\rightarrow W$ is a holomorphic branched covering. In other words, the 
Hurwitz class of a topological connected sum between surfaces contains a 
holomorphic branched covering. 

Hence  any branched covering of the Riemann sphere, in general position and degree $d$,  
can be presented as the connected sum of $d-1$ copies of $z^2.$

If  $R_1(z)=z^2$ and $R_{2}(z)=z^{3}$,
then $R_{1}\#R_{2}$ is a degree $4$ branched self-covering of the topological sphere.

The Figure \ref{manifolds} illustrates yet another example, in this case, it shows the connected
sum of rational maps in general position of degree $3$ and $4$ . 

\begin{figure}[ht]\centering
\scalebox{.65}{\input{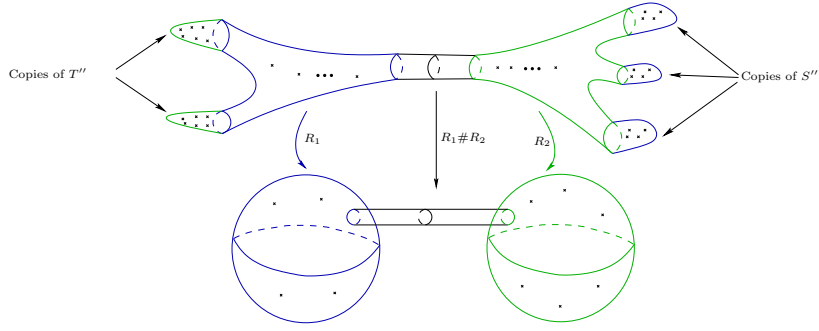}}
 \caption{The connected sum of rational maps $R_1$ and $R_2$, of degree $3$ and $4$ respectively. The dots depict the respective critical values below and their preimages above.}\label{manifolds}
 \end{figure}
 
 \subsection{Pinching.}

 We can recover the factors of the connected sum $R_1 \# R_2$  by a pinching procedure as follows.
 
 Let us note that if  $R_1\# R_2:S_1\rightarrow S_2$ is a connected sum of coverings with $deg(R_1),deg(R_2)\geq 2$, then $R_1\# R_2:U\rightarrow 
 W$ is a finite degree  covering between hyperbolic surfaces where $W=S_2\setminus CritVal(R_1\#R_2)$ and $U=S_1\setminus (R_1\#R_2)^{-1}(CritVal(R_1\#R_2))$. In other words, $R_1\#R_2$ always has a canonical representative whenever $deg(R_1), deg(R_2)\geq 2.$  Let $G, \Gamma_2$ be the Fuchsian 
 groups uniformizing  the surfaces $U$ and $W$ in the unit disk, respectively. Let 
 $\pi_U,\pi_W:\mathbb{D}  \rightarrow U,W$ be the respective uniformizing projections, 
 then there exists $\alpha $ a M\"obius automorphism of $\mathbb{D}$ satisfying  $R_1\# 
 R_2\circ\pi_U=\pi_W \circ \alpha$ and such that the subgroup $\Gamma_1=\alpha G \alpha^{-1} 
 < \Gamma_2$ has finite index. 
 
 We will say that the fixed pair of groups $\Gamma_1<\Gamma_2$ and the inclusion map uniformizes  $R_1\#R_2.$

On the other hand, the pair $\Gamma_1<\Gamma_2$ acts on $\mathbb{D}^*$ where  
$\mathbb{D}^{*}=\overline{\C}\setminus \overline{\mathbb{D}}$ and defines an orbifold covering map $Q:
U^*\rightarrow W^*$, where $U^*=\mathbb{D}^*/\Gamma_1$ and $W^*=\mathbb{D}^*/\Gamma_2$ are the 
anticonformal copies of $U$ and $W$ respectively, and such that $Q$ is an
anticonformal copy of $R_1\# R_2:U\rightarrow W.$

Let $C$ be the unique simple closed geodesic in the isotopy class of $[\alpha,
\beta]\subset W$. According to Theorem \ref{th.pinching}:

\begin{itemize}
\item There exists a sequence $f_k$  of  quasiconformal automorphisms of the Riemann sphere holomorphic outside the unit disk such that the groups 
$\Gamma_{2,k}=f_k\circ \Gamma_2\circ f_k^{-1}$ are  quasifuchsian groups converging 
to a Kleinian group $\Gamma_{2,\infty}$ with an invariant component of $\Omega(\Gamma_{2,\infty})$.
\item The surface $S(\Gamma_{2,\infty})$ is homeomorphic to $(W\setminus C) \sqcup W^*$ and $C$ determines a pair of punctures on $S(\Gamma_2,\infty).$ Moreover, 
the homeomorphism can be chosen conformal outside a tubular neighborhood of $C.$
\end{itemize}

In other words $S(\Gamma_{2,\infty})$ is conformally equivalent to 
$$(S'\setminus \{x\}) \sqcup (T' \setminus \{y\}) \sqcup W^*$$ where 
$x$ and $y$ are the additional cusps determined by $C.$  Indeed the perforations $x$ and
$y$ are known as \textit{accidental cusps} (accidental parabolics) which appear in pinching
processes. Let $\Gamma_{1,\infty}$ be the respective limit of the groups
$f_k\circ \Gamma_1 \circ f_{k}^{-1}$, then $\Gamma_{1,\infty}\cong \Gamma_1$. Since $(R_1\# R_2)^{-1}(C)$ consists of finitely many simple closed geodesics on $U$, then according to Theorem \ref{th.pinching} the surface
$S(\Gamma_{1,\infty})$ is topologically equivalent to
$(U\setminus (R_1\# R_2)^{-1}(C))\sqcup U^*$, 
these equivalences can be chosen conformal outside tubular neigborhoods of the curves in $(R_1\# R_2)^{-1}(C)$. 
In conclusion,
$\Gamma_{1,\infty}<\Gamma_{2,\infty}$ induces a covering map $H:
S(\Gamma_{1,\infty})\rightarrow S(\Gamma_{2,
\infty})$  so that the restriction of  $H$ to the component associated to $U^*$ is in the conformal Hurwitz class
of $Q$. Among the restrictions to the other components of $S(\Gamma_{1,
\infty})$ which are coverings there are only two which 
have degree bigger than one, the other restrictions are univalent. The pair of non-univalent coverings belongs to the conformal Hurwitz class of the coverings $R_1$ and $R_2.$ 

The result of pinching the manifolds of the example depicted in Figure \ref{manifolds} is 
shown in Figure \ref{pinched}.

\begin{figure}[h]\centering
 \input{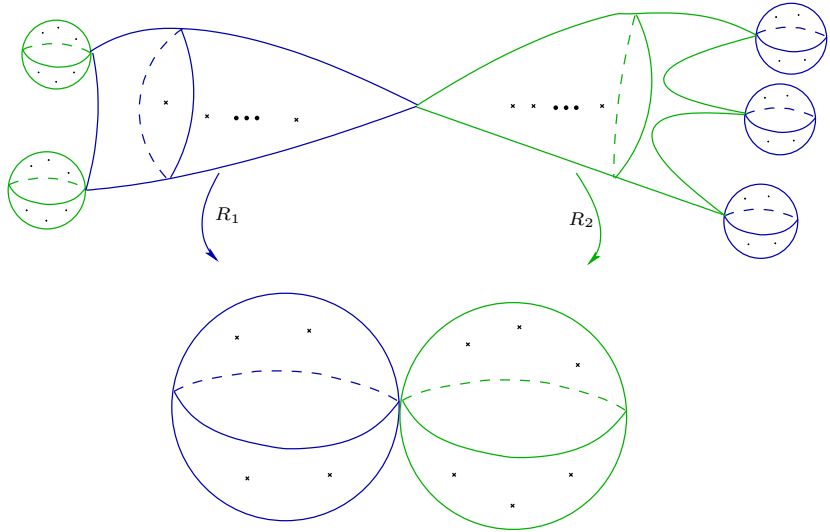}
 \caption{In this picture we applied pinching to the example in Figure \ref{manifolds}. }\label{pinched}
 
 \end{figure}

\subsection{Klein-Maskit combinations and the connected sums of coverings}
Now we are ready to prove Theorem \ref{TheoremA}.

\begin{proof}[Proof of Theorem \ref{TheoremA}] 
First, begin with two  single holomorphic covering maps $R_{1}:S \rightarrow S'$ and  
$R_{2}:T \rightarrow T'$  of degree $n$ and $m$ respectively, denote by 
${\Gamma}$ and ${\Gamma}'$ the respective Fuchsian uniformizing groups of $S$ 
and $S'$, and groups $G$ and $G'$ for the surfaces $T$ and $T'$ acting on the unit 
disk $\mathbb{D}$ so that $\Gamma<\Gamma'$ and $G<G'$ and the inclusion 
maps induces the covering $R_1$ and $R_2$, respectively. If 
$\mathbb{D}^*=\overline{\C}\setminus \overline{\mathbb{D}}$, then inclusion maps 
$(\Gamma,\mathbb{D}^*)\rightarrow (\Gamma',\mathbb{D}^*)$ and 
$(G,\mathbb{D}^*)\rightarrow (G',\mathbb{D}^*)$ define anticonformal copies 
of $R_1$ and $R_2$, respectively. Denote by $Q_1:S^*\rightarrow (S')^*$ and 
$Q_2:T^*\rightarrow (T')^*$ these anticonformal copies. We have that $[{\Gamma}':
{\Gamma}]=n$ and 
$[G':G]=m$, following the construction of previous section  we construct 
a covering map  of degree $n+m-1$, define $R_{0}:=Q_{1}{\#}Q_{2}$ thus $R_0$ maps $U$ to $W=(S')^*{\#}(T')^*$. \\

{\bf Group construction of $U$.} We  follow the topological construction above 
with the Klein-Maskit combination theorem. Let $\tau(z)=1/\overline{z}$ be the reflection and 
fix suitable round closed disks $D\subset \tau(F(\Gamma))$ and 
$\widetilde{D}\subset \tau(F(G))$, where $F(\Gamma)$ and $F(G)$ are fundamental regions for 
the actions of $\Gamma$ and $G$ in $\mathbb{D}$, respectively. Now let $h\in PSL(2,\C)$ be 
an element so that $h(\partial D)=\partial \widetilde{D}$ and $h$ maps the interior of $D$ onto the exterior of  $\widetilde{D}$. Then 
the pair of groups $\Gamma'$ and $h^{-1}\circ G'\circ h$ and the disks $D_1=D$ and 
$D_2=\overline{\C}\setminus D_1$  satisfy the conditions of the Klein-Maskit combination 
theorem.  Since $h^{-1}\circ G' \circ h$ is a M\"obius copy of $G'$ we can assume, by taking 
suitable M\"obius copies, that the groups $\Gamma'$ and $G'$ already satisfy the conditions 
of Theorem \ref{th.KleinMaskit}. By Theorem \ref{th.KleinMaskit}, the orbit space 
$S(\langle \Gamma',G'\rangle)$ is conformally 
equivalent to $S'\sqcup T' \sqcup ((S')^*\# (T')^*).$

Consider elements  $\{e,{\sigma}_{2},...,{\sigma}_{n}\} \subset {\Gamma}'$ and 
$\{e,g_{2},...,g_{m}\} \subset G'$  such that ${\Gamma}'={\Gamma} \cup ...\cup 
{\sigma}_{n}{\Gamma}$  and $G'=G \cup ...\cup g_{m}G$. Let $\Gamma_i=g_i\circ 
\Gamma'\circ g_i^{-1}$ be $m-1$ M\"obius copies of $\Gamma'$ and $G_j=\sigma_j \circ G' \circ 
\sigma_j^{-1}$ be $n-1$ M\"obius copies of $G'.$ Then  by an inductive application of the Klein-Maskit combination theorem the group $H=\langle \Gamma,G,
\Gamma_2,...,\Gamma_{m},G_2,...,G_n\rangle$ is isomorphic to
$$\Gamma* G * \prod_{i=2}^m \Gamma_i*\prod_{j=2}^n G_j.$$ 
Hence, the manifold $M(H)$ is a disk sum of the manifolds $M(\Gamma), M(G), M(\Gamma_i)$ and $M(G_j),$ where the latter are $n-1$ M\"obius copies of $M(\Gamma')$ and $m-1$ M\"obius copies of $M(G')$, respectively. 

The inclusion of $H$ in $\Gamma' * G'$ induces a holomorphic (orbifold) covering
$$\hat{\iota}:M(H)\rightarrow M(\langle 
\Gamma',G'\rangle)$$ which has finite degree, thus $H$ has finite index in $\langle 
\Gamma',G'\rangle$. 
Then the restriction $\hat{\iota}:S(H)\rightarrow S(\langle 
\Gamma',G'\rangle)$ is so
that there exist three surfaces $S_1,$ $S_2$ and $S_3$ in $S(H)$ 
where $deg(\hat{\iota}|_{S_j})>1$, the space $S_1\sqcup S_2 \sqcup S_3$ is conformally 
equivalent to $S\sqcup T \sqcup U$  and $\hat{i}(S_1\sqcup S_2 \sqcup S_3)$ is 
conformally equivalent to $S'\sqcup T' \sqcup W$. Moreover, the map 
$\hat{\iota}$  belongs to $CH(R_1,R_2,Q_1\#Q_2)=CH(R_1)\sqcup CH(R_2)\sqcup 
CH(Q_1\#Q_2).$ If $O\subset S(H) \setminus \{S_1\sqcup S_2 \sqcup S_3\}$, then 
$\hat{\iota}:=O\rightarrow \hat{\iota}(O)$ is an univalent holomorphic surjective map, even 
more $\hat{\iota}(O)$ is either $S'$ or $T'.$ So $\hat{\iota}$ is a non-simple holomorphic covering
containing single univalent components. 

In conclusion, $\hat{\iota}$ is a M\"obius morphism which defines a hyperbolic cobordism between 
the collections $R_1:S\rightarrow S'$, $R_2:T\rightarrow T'$, $Q_1'\#Q_2:U\rightarrow 
W$ and single univalent components.  

For the general case with three or more coverings $R_{i}:S_{i} \rightarrow S'_{i}$, 
$i=1,...k$, we proceed inductively. This finishes the proof. \end{proof}

Now what can we say about non-simple holomorphic coverings, we start with the following examples of uniformizable non-simple holomorphic coverings.

\begin{enumerate}

\item  Let $\Gamma$ be a geometrically finite Fuchsian group such that $\gamma\circ 
\Gamma \circ \gamma^{-1}=\Gamma$ for $\gamma(z)=\frac{1}{z}$. Then 
$G=\langle \Gamma, \gamma\rangle$ is isomorphic to an $HNN$-extension of $\Gamma$ 
and is a geometrically finite Kleinian group.  So that $M(G)=(\Omega(G)\cup B)/G$ is a geometrically finite hyperbolic orbifold with connected boundary which is a hyperbolic orbifold 
conformally equivalent to $\mathbb{D}/\Gamma.$ More, the inclusion 
$\Gamma< G$ induces a degree $2$ branched covering $\pi:M(\Gamma) 
\rightarrow M(G)$ so that $\pi(\partial M(\Gamma))=\partial M(G)$ and for any 
component $S\in \partial M(\Gamma)$ the restriction $\pi|_S:S\rightarrow \partial 
M(G)$ is a conformal equivalence. Now let $\Gamma_0< \Gamma$ be a 
subgroup of index $d.$  Then the inclusion $\Gamma_0$ in $\Gamma$ induces
a branched covering map $p:M(\Gamma_0)\rightarrow M(\Gamma)$ of degree $d$, so 
that for any component $S\subset \partial M(\Gamma_0)$ the restriction $p|_S$ is 
an orbifold covering map of degree $d.$ Then $\pi \circ p$ is a non-simple 
holomorphic covering of degree $d$. 

\item Let $\Gamma$  be a Kleinian group such that $\partial M(\Gamma)$ is 
connected, the  components of $\Omega(\Gamma)$ are simply connected with 
stabilizers of infinite index. Such groups are known as web-groups. Since 
geometrically finite Kleinian groups are also residually finite we can choose a 
subgroup $H<\Gamma$ of finite index such that  $\partial M(H)$ is disconnected, the map 
$\pi:\partial M(H)\rightarrow \partial M(\Gamma)$ given by the canonical holomorphic orbifold 
covering induced by the inclusion $H\subset \Gamma$ is non-simple with at least two 
components $S_1$ and $S_2\subset \partial M(H)$ such that 
$deg(\pi_{S_1})>deg(\pi|_{S_2}).$
\end{enumerate}

We have not found in the literature whether for any connected hyperbolic Riemann
surface $S$ there exists a web group $G$  with $\partial M(G)$ conformally
equivalent to $S$. However, M. Kapovich kindly pointed that this construction can be 
done using the Brooks Deformation theorem (see \cite{KapovichHyp}).

Given a holomorphic covering $(R,S,S')$. If $S'$ is connected then we say 
that $R$ is primitive. We call a primitive holomorphic covering uniformizable 
if there exists a pair of web groups $H< \Gamma$, with $H$ of finite index, so 
that the canonical holomorphic covering $\pi:\partial M(H)\rightarrow M(\Gamma)$ 
belongs to $CH(R,S,S')$. So far we have no examples of non-simple non-uniformazible 
holomorphic coverings. \\

\noindent \textbf{Connected sum of non-simple  coverings.} Given two primitive holomorphic 
coverings $(R_1,S_1,S'_1)$ and $(R_2,S_2,S'_2)$,  we construct a connected sum 
$R_1\#R_2$, in a similar but slightly different way as in Subsection \ref{subconnectedsums} as follows: for $i=1,2$ fix 
two open topological disks $t_i\subset S'_i$, not containing branching points of $R_i$, 
respectively,  together with a homeomorphism $\phi:\partial t_1 \rightarrow \partial t_2$.
Let $g_1$ be a single component of $R_1$, then the map $\phi\circ g_1$
defines a family of homeomorphisms from the components of $g_1^{-1}(\partial 
t_1)$ to $\partial t_2$. Let $S''_2=S'_2\setminus t_2$ and we glue the copies of 
$S''_2$ to the surface  $g_{1}^{-1}(S'_1\setminus t_1)$ along the family of  
homeomorphisms $\phi\circ g_1$ on $g_1^{-1}(\partial t_1)$. Using induction with 
respect to all single components of $R_1$ we construct a
 non-simple holomorphic covering $(\hat{R}_1,T,T')$ with $T'=S'_1\# S'_2$.
Now repeat the process for a single component $g_2$ of $R_2$ to get a non-simple 
holomorphic covering $(\hat{R}_2,W,W')$ with $W'=S'_1\#S'_2$ and finally we put
$R_1\#R_2=(\hat{R}_1,T,T')\sqcup (\hat{R}_2,W,W')$ which is a primitive
non-simple holomorphic covering. 

Let us note that the case  of holomorphic coverings over a surface of genus zero is
special in the following sense: let $R_1$ and $R_2$ be non-simple holomorphic 
coverings onto the Riemann sphere with finitely many points removed. If $g$ is a 
single component of $R_1$ then the induced single component $f$ of $R_1\# R_2$
belongs to $H(g)$ up to forgetting additional perforations on the source and the target 
surfaces. We call such primitive covering a \textit{primitive covering of genus zero}. 
This observation leads to the following lemma.

\begin{lemma}\label{lemma8}
Let $R_1:S_1\rightarrow S'_1$ and $R_2:S_2\rightarrow S'_2$ be holomorphic coverings 
over $S'_1$ and $S'_2$ which are Riemann spheres with finitely many points 
removed. Then we can choose for $i=1,2$ disks $t_i\subset S'_i$ and a gluing map 
$\phi:\partial t_1\rightarrow \partial t_2$ such that every single component of  $R_1\# 
R_2$ belongs to either $CH(R_1)$ or $CH(R_2)$ up to forgetting additional 
perforations. 
\end{lemma}

\begin{proof}Let $t\subset \overline{\C}$ be a round open disk such that $t\subset S'_1$ 
and not containing branching points of $R_1.$ Let $X\subset \overline{\C}$ be the set 
$\overline{\C}\setminus (S'_2\cup V(R_2))$, where $V(R_2)$ is the set of branching points 
of $R_2$, and $\gamma\in PSL(2,\C)$ is so that $\gamma(X)\subset t.$ Then for the 
coverings $R_1$ and $\gamma\circ R_2$ we choose $t_1=t$ and 
$t_2=\overline{\C}\setminus t.$ Taking $\phi=Id$ on $\partial t$ finishes the proof of the 
lemma.
\end{proof}

Another application of Klein-Maskit combination theorem and the arguments of the proofs of  Theorem \ref{TheoremA} and Lemma \ref{lemma8} allow us to improve 
Theorem \ref{TheoremA} as follows. 

\begin{theorem}\label{Theorem1prima} Let $R_1$ and $R_2$ be two primitive 
uniformizable  genus zero holomorphic coverings then there exists a primitive uniformizible
genus zero holomorphic covering $Q$ so that for every single component $q$ of $Q$ there 
exist $r$, which is either a single component of $R_1$ or a single component of $R_2$, so that $q\in CH(r)$ after, perhaps, forgetting additional perforations.
\end{theorem}

\begin{proof} Let $\Gamma_1$ and $\Gamma_2$ be web-groups uniformizing 
$R_1$ and $R_2$. As in Theorem \ref{TheoremA} we apply the Klein-Maskit 
combination theorem to $\Gamma_1$ and $\Gamma_2$ and construct a 
finite index subgroup of $\Gamma_1*\Gamma_2$ compatible to the connected sum
$R_1\# R_2.$
\end{proof}

Remark. \textit{The Klein-Maskit theorems are  generalized to Kleinian groups in higher 
dimensions, so our Theorem \ref{TheoremA} generalizes in that setting as well.}

\section{Proof of Theorem \ref{Th.Hurwitz}}\label{ProofHurwitz}
Theorem \ref{Th.Hurwitz} is a direct application of the Bers's simultaneous uniformization 
theorem.

\begin{proof}[Proof of Theorem \ref{Th.Hurwitz}.] 

Assume that $R_1$ and $R_2$ are two symmetric holomorphic coverings forming a 
simple  hyperbolic cobordism, then the respective geometrically finite Kleinian groups 
$\Gamma_1$ and $\Gamma_2$ are quasifuchsian.
We can assume that 
$\Gamma_1<\Gamma_2$. Since every quasifuchsian group is quasiconformally equivalent to 
a Fuchsian group then $\Gamma_2$ admits an orientation reversing quasiconformal 
involution $\tau:\overline{\C}\rightarrow \overline{\C}$ commuting with $\Gamma_2$, interchanging  
components of $\Omega(\Gamma_2)$ and which is the identity on the limit set 
$\Lambda(\Gamma_2)$. Hence, $\tau$ commutes with $\Gamma_1.$ Since $R_1$ and 
$R_2$ are symmetric, then $R_1$ and $R_2$ belong to the same Hurwitz class. 

Assume that two symmetric holomorphic coverings $R_1$ and $R_2$ belong to the 
same Hurwitz class then we can construct cobordisms between $R_1$ and  an
anticonformal copy of $R_1$, together with  a cobordism between $R_2$ and an anticonformal 
copy of $R_2$. The homeomorphisms $\phi$ and $\psi$ associated to $R_1$ and $R_2$ 
allows us to glue the given cobordisms along the anticonformal copies to get a 
cobordism between $R_1$ and $R_2.$
\end{proof}

What follows is an example of a topological cobordism between the simplest rational maps.

 Let us consider the convex combination between $z^2$ and $z^3$:
 $$f_t(z)=(1-t)z^2+t z^3$$ for $t\in I=[0,1].$ Then $f_t$ defines a rational endomorphism 
 $F$ of $3$-manifolds $X=\C\times I$ by the formula
 $$F(z,t)=(f_t(z),t).$$ Then
 \begin{itemize}
  \item[i)] The map $F$ is not a branched self-covering of $X$.
  \item[ii)] Let $X_{t_1,t_2}=\C\times [t_1,t_2]$, then the restriction 
  of $F$ to $X_{t_1,t_2}$ is a branched self-covering of $X_{t_1,t_2}$ for
  $0<t_1<t_2<1.$
  \item[iii)] If $0<t_1<t_2<1$, then for $t_1\leq t\leq t_2$ the real polynomial $f_t$ is 
  in general position and the sets of the critical values $\{v_1\}, \{v_2\}$ of $F$ in 
  $X_{t_1,t_2}$ forms two embedded arcs connecting the boundaries of 
  $X_{t_1,t_2}.$ Moreover, for each $i$ the set $F^{-1}(v_i)$ consists of 
  two curves $\alpha_{i,j}$, with $j=1,2$, one is mapped homeomorphically onto the image by $F$, while the other consists of the critical points. Let 
  $M=X_{t_1,t_2}\setminus (v_1\bigcup v_2)$ and $M=F^{-1}(M')\subset X_{t_1,t_2}$ be two three manifolds, then $F:M\rightarrow M'$ 
  is a covering. The manifold $M$ is homeomorphic to the product of 
  the five punctured sphere times a closed interval and the manifold $M'$ is homeomorphic
  to the three punctured sphere times a closed interval.
 Since $f_{t_i}$ are symmetric maps and $f_{t_2}\in H(f_{t_1})$, then 
 by Theorem \ref{Th.Hurwitz} the maps  $f_{t_1}$ and $f_{t_2}$ are simply cobordant,
 then on $M$ and $M'$ there are hyperbolic structures depending on the 
 extremes $t_1$ and $t_2$ such that $F$ is a M\"obius morphism on $M.$
 This means that  for $0<t_1<t_2<1$ we can define two 
 hyperbolic $3$-dimensional orbifold structures on which the map $F$ is 
 a M\"obius morphism.  
 
 \item[iv)] The endomorphism $F:X\rightarrow X$ is a Hausdorff limit (this is a 
 particular case of Gromov-Hausdorff  limit) of M\"obius 
 morphisms $F:X_{t_1,t_2}\rightarrow X_{t_1,t_2}$ for $t_1\rightarrow 0$ and
 $t_2\rightarrow 1$.
 \item[v)] If $t\neq 0,1$, then we can obtain two functions $f_0$ and $f_1$ from 
 the map $f_t$ using a pinching procedure with respect to peripheral curves.
 For convenience of the reader we sketch this procedure when $f_1(z)=z^3.$
 Fix a $t\neq 0, 1$ and a Jordan curve $\gamma\in \C$ so that $\alpha=f_{t}^{-1}
 (\gamma)$ is a connected Jordan curve.  Then the finite critical values belong to the 
 interior of $\gamma.$
 Let $A(\gamma)$ be an annular neighborhood of $\gamma$ so that 
 $A(\alpha)=f^{-1}_t(A(\gamma))$ is an annular neighborhood of $\alpha$ ad 
 $f_t:A(\alpha)\rightarrow A(\gamma)$ is a covering of degree $3.$ Let $\nu_j$ be 
 a sequence of Beltrami differentials supported in $A(\gamma)$ as constructed in 
 Section \ref{subsectionpinching} and consider the extension of each $\nu_j$ on $\overline{\C}$  
 by zero outside $A(\gamma)$. Let $\mu_j(z)=v_j(f_t)\frac{\overline{f'_t}}{f'_t}$ 
 be the pull-back of $\nu_j$ with respect to $f_t$. Let $\phi_j$ and $\psi_j$ 
 be solutions of the Beltrami equation for $\mu_j$ and $\nu_j$ respectively, 
 with $\phi_j(0)=\psi_j(0)=0,$  $\phi_j(\infty)=\psi_j(\infty)=\infty,$ and $\phi'_t(\infty)=\psi'_j(\infty)=1$, then
 $p_j=\psi_j\circ f_t \circ \phi_j^{-1}$ are polynomials of degree $3.$ Moreover, 
 $\phi_j$ forms a family of univalent normalized holomorphic functions defined on 
 a neighborhood of infinity $V,$ then $\psi_j$ also forms a holomorphic family 
 on $U=f_t^{-1}(V).$ After taking a suitable subsequence we can assume that
 $\phi_j\rightarrow \phi_0$ and $\psi_j\rightarrow \psi_0$ converge uniformly on 
 compact subsets of $U$ and $V$, respectively. Moreover, $\phi_0$ and $\psi_0$ are non 
 constant functions. Hence $p_j$ converges to a degree $3$ polynomial $p_0$ and, 
 even more, $p_0|_{\phi_0(U)}=\psi_0\circ f_t \circ \phi_0^{-1}|_{\phi_0(U)}$. We claim that 
 $p_0(z)=z^3,$ otherwise, $p_0$ has a critical value $v_0\neq 0,\infty.$ Thus $p_j$ 
 also has a finite critical value $v_j$ converging to $v_0.$ But 
 $v_j=\psi_j(\frac{4}{27}{(1-t)^3}{t^2})$ belongs to the bounded component of 
 $\C\setminus \{\psi_j(A(\gamma))\}.$ Since $\psi_j$ converges to $\psi_0$ and the 
 moduli of the annuli $\psi_j (A(\gamma))$ converges to $\infty, $ we have a contradiction to 
 $v_0\neq 0.$ Thus $p_0(z)=z^3$ as claimed.

 In the case where $f_0(z)=z^2$, we consider a curve $\gamma$ closed to $0$ so that $f_t^{-1}
 (\gamma)$ is the union of two curves $\alpha$ and $\beta$. Here $f_t$ is a degree 
 $2$ covering on $\alpha$, and $\beta$ belongs to the exterior of $\alpha.$ Now we take the 
 normalization of $\phi_j$ and $\psi_j$  given by $\phi_j(0)=\psi_j(0)=0,$ $\phi'_j(0)=\psi'_j(0)=1$
 and $\phi_j(\infty)=\psi(\infty)=\infty$ and proceed as above.
 
 \item[vi)]Now  on $X_{t_1,t_2}$ with $0<t_1<t_2<1$ 
 we put a quasiconformal deformation converging on the components of the boundary 
 $\partial X_{t_1,t_2}$ to $z^2$ and $z^3$, respectively to get a limit.
 This limit seems to be a sort of double  limit converging in Hausdorff topology to a 
 non-uniformizable object. This would give a  non-geometric completion of the respective 
 Teichm\"uller space.  We suspect that any limit 
 of this type belongs to $H(F,\C\times [0,1],\C\times [0,1]).$
 
 \item[vii)] Let $J(f_t)$ be the Julia set of $f_t$ then $J(F)= \bigcup_{t\in [0,1]} (J(f_t),t) \subset \C \times [0,1]$, is a closed set 
 completely invariant under $F$. Computer experiments suggest that
 $\C\times I \setminus J(F)$ consists of two components. The one containing the set 
 $\{0\}\times [0,1]$ is simply connected and the other one has infinitely generated 
 fundamental group. The experiments also suggest that the set $J(F)$ is a non-locally connected embedding 
 of $\mathbb{S}^1\times [0,1]$ with infinitely many cusps accumulating to
  a compact subset of  the interior of $\mathbb{S}^1\times [0,1]$. 
 \end{itemize}
 
\section{Proof of Theorem \ref{TheoremB}}\label{proofTheoremB}

As the previous discussion showed, the construction of cobordisms between 
rational maps often involves the introduction of single univalent components.
In this chapter we look for a construction avoiding  these components. 
We call this type of construction a \textit{pure cobordism}.

Let $\mathbb{D}_\infty$ be the completion of $\C$ by adding the circle at infinity 
$\infty\cdot \mathbb{S}^1$. Then any monic holomorphic polynomial $P$ of degree $d$ 
can be extended to the circle at infinity by the formula 
$\infty\cdot e^{i\theta}\mapsto \infty\cdot e^{d i \theta}$ so that  $P$ defines 
a branched self-covering $\hat{P}$ of $\mathbb{D}_\infty$ and $\hat{P}|_{\partial 
\mathbb{D}_\infty}(z)=z^d$. Now we identify $\mathbb{D}_\infty$ with the unit disk $\mathbb{D} \subset \C.$
Let us consider two monic polynomials $P_1$ and $P_2$, 
acting on $\C$, then define a finite degree branched covering $F:\overline{\C}\rightarrow 
\overline{\C}$ as follows $$
 F(z) = 
  \begin{cases}
    \hat{P}_1(z), & \text{for } z\in \mathbb{D} \\
    \gamma \circ\hat{P}_2 \circ \gamma(z), & \text{for } z \in \overline{\C}\setminus \mathbb{D}, \gamma(z)=\frac{1}{\overline{z}}.
  \end{cases}
 $$
 
In holomorphic dynamics the map $F$ is known as the \textit{formal mating} of the monic 
polynomials $P_1$ and $P_2$ (see for example  \cite{MilnorMating} or \cite{TanMating}). 

Now we need the following theorem. 

\begin{theorem}\label{Purecobordism}Given two polynomials $P_1$ and 
$P_2$ of the same degree $d>2$ with canonical holomorphic representatives.
Assume one of them, say $P_2$ is symmetric, then there exists a rational map 
$R$ such that the following holomorphic coverings form a hyperbolic cobordism:
\begin{itemize}
 \item $P_1:\C\setminus (P_1^{-1}(V(P_1)))\rightarrow \C\setminus V(P_1)$,
 \item $P_2:\C\setminus (P_2^{-1}(V(P_2)))\rightarrow \C\setminus V(P_2)$,
 \item $R:\C\setminus (R^{-1}(V(R)))\rightarrow \C\setminus V(R)$.
\end{itemize}
\end{theorem}

\begin{proof}
 Given two polynomials $P_1$and $P_2$ as in the statement of the theorem,  
 we can assume that $P_1$ and $P_2$ are monic. Let us consider their formal mating $F$. Let 
 $Q\in H(F,\overline{\C},\overline{\C})\cap Rat$ be a rational map. 
 First note that since $F^{-1}(\partial \mathbb{D})=\partial \mathbb{D}$
 then there is a Jordan curve $\delta\subset \overline{\C}$ such that
 $Q^{-1}(\delta )$ is a connected Jordan curve. If $V(Q)$  is the set of critical values 
 of $Q$ then  $Q$ is a holomorphic covering of finite degree from $S_1=\overline{C}\setminus Q^{-1}(V(Q))$ 
 to $S_2=\overline{\C}\setminus V(Q)$.
 
 Let $\Gamma_1<\Gamma_2$ be two finitely generated 
 Fuchsian groups uniformizing $Q:S_1\rightarrow S_2$ in $\mathbb{D}$. Then we 
 can assume that the map $R:\mathbb{D}^*/\Gamma_1\rightarrow \mathbb{D}^*/ 
 \Gamma_2$ induced by the inclusion is a rational map satisfying 
 $R(z)=\overline{Q(\overline{z})}$. We claim that the triple $(R,P_1,P_2)$ forms a cobordant
 family. 
 
 Indeed, let us apply a pinching procedure to the covering $Q$, with respect to the curve 
 $\delta$.  Then by Theorem \ref{th.pinching}, we obtain a Kleinian
 group $\Gamma_{2,\infty}$ with an isomorphism $\rho:\Gamma_2\rightarrow 
 \Gamma_{2,\infty}$ so that $\Gamma_{2,\infty}$ is a geometrically finite 
 functional group with $S(\Gamma_{2,\infty})=
 S_2'\sqcup T'_2 \sqcup T''_2$ where $S'_2$ is anticonformally equivalent to $S_2$ 
 and $T'_2$ and $T''_2$ are finitely punctured Riemannian spheres, each of them contains
 only one accidental cusp, that is determined by $\delta$.
 
On the other hand, we have a geometrically finite functional group $\Gamma_{1,\infty}$ satisfying $\Gamma_{1,\infty}<\Gamma_{2,\infty}$. In particular, since $R^{-1}
 (\delta)$ is connected then $\Gamma_{1,\infty}$ also is presented by pinching the group $\Gamma_1$ with respect to the curve $R^{-1}(\delta)$ and contains only one conjugacy class 
 of an accidental parabolic element. Hence $S(\Gamma_{1,\infty})=S'_1\sqcup T'_1 \sqcup T''_1$, where $S'_1$ is an anticonformal copy of $S_1$ and the surfaces $T'_1$ and $T''_1$ are finitely punctured spheres containing only one accidental cusp determined by  $R^{-1}(\delta).$  This implies that 
 the map $\alpha$ induced by inclusion is a holomorphic simple covering from 
 $S(\Gamma_{1,\infty})$ to $S(\Gamma_{2,
 \infty})$ so that $\alpha|_{S'_1}:S'_1\rightarrow S'_2$ belongs to
 $H(R)$. As the preimage of an accidental cusp is an accidental cusp, then  the  maps 
 $\alpha|_{T'_1}$ and $\alpha|_{T''_1}$ are in the Hurwitz conformal 
 classes of  some  polynomials, say $Q_1$ and $Q_2$ respectively. By construction $Q_1\in 
 H(P_1)$ and $Q_2\in H(P_2)$ with orientation reversing 
 homeomorphisms $\phi_2$ and $\psi_2$. Since  $P_2$ is 
 symmetric then the homeomorphisms $\phi_2$ and $\psi_2$ also can be chosen to be 
 quasiconformal orientation preserving homeomorphisms. If $\phi_i\circ Q_i=P_i\circ \psi_i$, 
 then let $\mu$ be the Beltrami differential on $S'_2\sqcup T'_2\sqcup T''_2$ given in local coordinates by 
 $$
 \mu(z) = 
  \begin{cases}
    \frac{\overline{\partial} \phi_1}{\partial\phi_1}(z) , & \text{on } T'_2\\
    \frac{\overline{\partial} \phi_2}{\partial\phi_2}(z), & \text{on } T''_2\\
    0, &\text{ on } S'_2.
  \end{cases}
 $$
 and let $\nu$ be the pull-back of $\mu$ by the orbit projection 
 $$\pi_2:\Omega(\Gamma_{2,\infty})\rightarrow \Omega(\Gamma_{2,
 \infty})/\Gamma_{2,\infty}.$$
 If $f_\nu$ is a solution of the Beltrami equation with respect to $\nu$, then the 
 groups $G_1=f_\nu \circ \Gamma_{1,\infty} \circ f^{-1}_{\nu}$ and $G_2=f_\nu 
 \circ \Gamma_{2,\infty} \circ f^{-1}_{\nu}$ satisfies our claim and finish the proof 
 of the theorem. 
 \end{proof}

 Let us note that, by Theorem \ref{th.pinching}, each of the manifolds $M(G_1)$ and $M(G_2)$, constructed in 
 the proof of the Theorem \ref{Purecobordism}, is homeomorphic to a set $U$ which is the complement of 
two open round $3$-dimensional balls $B_1$ and $B_2$ in the unit
 ball $B$ with finitely many embedded arcs, connecting the boundary components
 of $\partial U$, removed.  See Figure \ref{figballs.arcs}.
  
  \begin{figure}[h]\centering
 \scalebox{.7}{\input{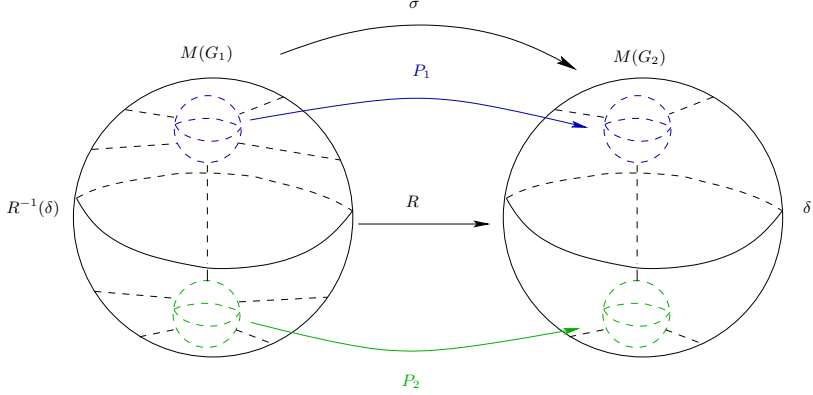}}
\caption{A sketch of the manifolds $M(G_1)$ and $M(G_2)$.}\label{figballs.arcs}
 \end{figure}
 
 The map $\sigma:M(G_1)\rightarrow M(G_2)$ induced by the inclusion $G_1<G_2$ 
 can be extended on $U$ as a finite degree branched covering 
 $\sigma^*:U\rightarrow U$ so that in suitable coordinates 
 $\sigma^*|_{\partial U}\in CH(R,\overline{\C},\overline{\C})\cup CH(P_1,\overline{\C},\overline{\C})\cup 
 CH(P_2,\overline{\C},\overline{\C})$. In other words, in $U$ there are two M\"obius orbifold 
 structures which makes $\sigma^*$ a  M\"obius morphism.
 
 Now if $\tau$  is the reflexion of $S^3$ with respect to the unit sphere, then 
 on $W=U\cup \tau(U)$ we can extend $\sigma^*$ to a finite degree self-covering
 of $W$ by putting 
 
  $$
 \Sigma(z) = 
  \begin{cases}
    \sigma^*(z) & \text{ on } U \\
    \tau\circ \sigma^* \circ \tau(z), & \text{on } \tau(U).
  \end{cases}
 $$
 
 Then $\Sigma$ serves as a topological cobordism between $P_1\sqcup P_2$ and its
 anticonformal copies. In what follows, we will show that on $W$ there are two 
 M\"obius structures under which $\Sigma$ is a M\"obius morphism.  The idea is to apply the arguments of the proof of Brooks's orbifold 
 deformation theorem to put
 on $M(G_2)$ a non-orientable uniformazible orbifold structure $\omega$ so that the 
 component $S_0\subset \partial M(G_2)$ corresponding to the unit sphere consists of 
 the ``interior points'' in the orbifold structure $\omega$ and other components 
 of $\partial M(G_2)$ equipped with $\omega$ are  M\"obius equivalent to the previous  
 structure.
 
 Then there exists a degree two covering $\beta:X\rightarrow (M(G_2),\omega)$ such 
 that $X$ is an orientable hyperbolic $3$-orbifold with four boundary components which are 
 mapped univalently by $\beta$ onto $\partial M(G_2)\setminus S_0.$
 
 The following statement is the main lemma of this section. Which is an application of 
 arguments of the proof of Brooks's orbifold deformation theorem (see  \cite{KapovichHyp}).
 
 \begin{lemma}\label{lemma.kap}
  Let $\Gamma$ a geometrically finite Kleinian group without torsion. Let $T\subset S(\Gamma)$ be a
  proper subcollection of surfaces, then there exist
  a geometrically finite Kleinian group $G$ such that $S(G)$ is conformally 
  equivalent to $T\sqcup T^*$ where $T^*$ is an anticonformal copy of $T.$
 \end{lemma}

 \begin{proof}
 Let $T'$ be the complement of the collection $T$ in $S(\Gamma)$. Our goal is to 
 destroy the non-empty collection $T'$. First, 
 assume that there exists a round disk pattern $K$ covering  just $T'$. Let $\Gamma_K$ be the 
 group generated by $\Gamma$ and the reflections with respect to all the disks projecting onto the elements of  $K$ as in the 
 discussion before Theorem \ref{th.Theorem6}. Then by Theorem \ref{th.Theorem6} the hyperbolic
 orbifold $M(\Gamma_K)$ has as underlying space the manifold $M_K$, where $M_K$ is homeomorphic to $M(\Gamma)$. Let $G_K<\Gamma_K$ be the subgroup
 of orientation preserving elements of $\Gamma_K$, then $G_K=\Gamma_K\cap 
 PSL(2,\C)$ is a normal order two subgroup of $\Gamma_K$ containing $\Gamma.$
 By the remark after Theorem \ref{th.Theorem6}, if a component $\Omega_0\in \Omega(\Gamma)$ 
 covers an element of $T$ then the stabilizer of $\Omega_0$ in $G_K$ 
 coincides with the stabilizer of $\Omega_0$ in $\Gamma.$ Then $M(G_K)$ 
 admits an anticonformal involution $\tau$ and  $M(\Gamma_K)=M(G_K)/\tau.$ 
 Hence $\partial M(G_K)=T\sqcup T^*$, where $T^*$ is an anticonformal copy of $T$.
 
 To finish the proof we have to justify the existence of the pattern $K$. By Theorem \ref{th.geofintor} there exists a quasiconformal homeomorphism $h$ and a group $\Gamma_h= h\circ \Gamma \circ h^{-1}$ admitting a disk pattern covering the whole surface $S(\Gamma_h)$.
 
 Let us consider a subpattern 
 $K'\subset K$ precisely covering the collection $h(T')\subset S(\Gamma_h)$ and
 construct a group $G_h$ which uniformizes the surfaces $S(G_h)=h(T)\sqcup 
 (h(T))^*$, taking a suitable quasiconformal deformation  for $G_h$ we obtain 
 a group $G$ as claimed.
 \end{proof}

 The lemma above allows to produce a hyperbolic orbifold structure on a double of the 
 manifold $M(\Gamma)$ with respect to a complementary collection $T'.$ In other 
 words, with this lemma one can endow an orbifold structure on the manifold $W.$
 The following theorem shows that we can put other orbifold hyperbolic structures on 
 $W$ in such a way that the map $\Sigma$ becomes a M\"obius morphism between these 
 structures. 
 
 \begin{theorem}\label{thIndbas} The family of canonical holomorphic 
representatives of any  collection of four polynomials in general position of 
the same degree $d>2$ forms a hyperbolic cobordism.
\end{theorem}

 \begin{proof}
 Take any pair of polynomials $P_1$ and $P_2$ from the given four. By the Theorem 
 \ref{Purecobordism}, there exists a rational map $R$ such that $R\sqcup P_1 \sqcup 
 P_2$ forms a cobordant family of coverings. Let $\Gamma_1< \Gamma_2$  be 
 the Kleinian groups realizing this cobordism, that is the 
 $\alpha: M_{\Gamma_1}\rightarrow M_{\Gamma_2}$ so that $\alpha$ maps  
 $S(\Gamma_1)=V_0\sqcup V_1 \sqcup V_2$ onto $S(\Gamma_2)=U_0\sqcup 
 U_1\sqcup U_2$ and $(\alpha, \partial M(\Gamma_1),\partial M(\Gamma_2))$
 belongs to the Hurwitz class $CH(R,P_1,P_2)$. Let $H_2$ be the geometrically finite group 
 with $S(H_2)=(U_1\sqcup U_2)\sqcup ((U_1)^*\sqcup (U_2)^*)$ given by 
 Lemma \ref{lemma.kap}, here again $(U_1)^*$ and $(U_2)^*$ are anticonformal copies of $U_1$ and $U_2$, respectively. 
 
 We claim that there exists a finite index subgroup $H_1<
 H_2$ with $S(H_1)=(V_1\sqcup V_2)\sqcup ((V_1)^*\sqcup (V_2)^*)$
 and a projection induced by inclusion of groups $$\beta: M(H_1)\rightarrow M(H_2)$$
 is so that $\beta|_{V_i}$ is conformally equivalent to $P_i$, for $i=1,2$, and 
 $\beta|_{(V_i)^*}$ are anticonformal copies of $P_i$ respectively.
 Indeed, by Brooks's orbifold deformation theorem as used in Lemma \ref{lemma.kap}
 can assume that the group $\Gamma_2$ admits a pattern $K$  which 
 covers only the surface $U_0$.
 
 Since $\alpha$ is  a M\"obius morphism then $K'=\alpha^{-1}(K)\subset V_0$ is also 
 a pattern on $S(\Gamma_1)$ completely covering just the surface $V_0$. 
 Hence the group $G_1$, generated by $\Gamma_1$
 and the reflections with respect to the boundaries of all round disks which projects on all elements of $K'$, is 
 a finite index subgroup of the group $G_2,$ where $G_2$ is generated by $\Gamma_2$
 and the reflections with respect to the boundaries of all disks projecting on all elements of $K.$ In fact,
 these families of disks for $\Gamma_1$ and $\Gamma_2$ coincide.
 
 Therefore the orientation preserving subgroups 
 $H_1=G_1\cap PSL(2,\C)$ and $H_2=G_2 \cap PSL(2,\C)$ are geometrically finite
 Kleinian groups such that $H_1$ has finite index in $H_2.$  Then the groups $H_1< H_2$ are the desired groups as claimed. 
 
Let us note that the anticonformal copies of $P_1$ and $P_2$
 belong to the Hurwitz classes of the polynomials $P_3$ and $P_4$. Indeed all polynomials in general 
 position, of the same degree, belong to the same Hurwitz class. Hence, after 
 a suitable quasiconformal deformation of the pair $H_1$ and $H_2$ we complete the proof of the 
 Theorem.
 \end{proof}

Choose a surface  $S\subset S(G_2)$ together with a pattern precisely covering $S$. 
Repeating the construction above, we construct a cobordism between the canonical 
representatives of six polynomials of the same degree in general position. The iteration of 
this procedure shows the following statement: 

If the canonical holomorphic representatives of a collection of polynomials $P_1,...,P_n$ in 
general position of the same degree $d>2$ forms a hyperbolic cobordism, then for every 
$k<n$, the canonical holomorphic representatives of every collection of $2(n-k)$ 
polynomials in general position  and of degree $d$ forms a hyperbolically cobordant 
family. An induction argument over Theorem \ref{thIndbas} and Theorem \ref{Th.Hurwitz} 
completes the proof of Theorem \ref{TheoremB}. 

Remark. Unfortunately, we were not able to prove the following desirable statement:

\textnormal{Every finite collection of non-univalent rational (polynomial) maps forms a cobordant family.}

This is not clear even in the case of a single rational map $R$ (for more details on this 
problem see \cite{CabMakSie}). 

But if we drop the geometrically finiteness condition in the definition of cobordism then, as it was shown in \cite{CabMakSie}, in 
the case of a single rational map $R$  
the statement above is always  true, the corresponding uniformizing group is totally 
degenerated and the respective 
M\"obius morphism is Hurwitz equivalent to the radial extension of $R$ in the unit $3$-dimensional ball.

\section{On conformal Hurwitz classes and sandwich semigroups. }

According to the discussion above it is interesting to know when two given 
rational maps belong to the same conformal or anticonformal Hurwitz class or, better, 
when these maps are conformally or anticonformally conjugated. It turns out that the answer is 
purely algebraic and does not requires any dynamical information. We give a precise answer
using sandwich products induced by the given rational maps. Also we 
suggest another point of view on Hurwitz classes of rational maps as 
minimal representation spaces of semigroups of holomorphic correspondences 
associated to the given holomorphic coverings.  
 
The results in this section develop ideas in \cite{CabMakSie}, and the Schreier 
representation of semigroups as treated in \cite{CaMakPl}. We start with a brief
introduction to Schreier representations. 

\begin{definition} Let $X$ be a topological space and $End(X)$ be the semigroup 
of continuous endomorphisms of the space $X$. Then 
\begin{enumerate}
 \item $End(X)$ is a topological semigroup.
 \item $X$ canonically embeds in $End(X)$ as the ideal $I$ of constant endomorphisms.
\item $I$ is the unique minimal bi-ideal (left and right) consisting of idempotents.
\item (Schreier Lemma) Let $G< End(X)$ be a subsemigroup with $G\cap I=A\neq 0$
and $\rho: G\rightarrow End(X)$ be a homomorphism then there exists a map 
$f:A\rightarrow X$ such that $f(g(x))=\rho(g)(f(x))$ for all $x\in A$ and $g\in G$, here 
$f(x):=\rho(x)$.
In other words, every homomorphism is generated by a map. Even more,  the map $f$ is continuous if 
and only if $\rho$ is continuous, and $f$ is a homeomorphism if and only if $\rho$ is a 
continuous isomorphism onto its image. We say that $\rho$ is orientation preserving or non-orientable, depending on whether $f$ has the corresponding property.

\item Let $f:Y\rightarrow X$ be a continuous map between topological spaces, then the set
$G$ of all continuous maps $g:X\rightarrow Y$ can be transformed into a semigroup with the  product: $$g_1\ast_{f} g_2=g_1\circ f \circ g_2.$$ This product is called the sandwich 
product with respect to $f$ and $G_f=\langle G, \ast_f\rangle$ is called the sandwich 
semigroup. If $f$ is not invertible then $G_f$ does not contain a unit. 
 \end{enumerate}
\end{definition}

The following theorem appears in \cite{CaMakPl}. For convenience we include the proof.

\begin{theorem}
Let $R_1:\overline{\C}\rightarrow \overline{\C}$ and $R_2:\overline{\C} \rightarrow \overline{\C}$ be two 
rational maps, and $G_1$ and 
$G_2$ be sandwich semigroups of rational maps with respect to $R_1$ and $R_2$, respectively.
If $\rho:G_1\rightarrow G_2$ is an isomorphism, then there exist an element $\gamma\in 
PSL(2,\C)$, and a bijection $\phi:\overline{\C}\rightarrow \overline{\C}$ so that 
$\rho(R)= \phi\circ  R \circ \phi^{-1} \circ \gamma$ for every rational map $R$.
\end{theorem}

We say that the homomorphism $\rho$ is orientation preserving or non orientable depending 
on whether $\phi$ has the same property.

\begin{proof}
Let $f=\rho|_{\C}$ be the restriction of $\rho$ to the constants then $f(\overline{\C})\subset 
\overline{\C}$. Indeed for a suitable constant 
$c\in \overline{\C}$ then $c\ast_{R_1} Q=c$  and hence 
$\rho(c)\ast_{R_2} \rho(Q)=\rho(c)$ for every rational map  $Q$. Since $\rho$ is an isomorphism $\rho(Q)$, then $\rho(c)$ is a constant. Also $f$ is a bijection. 
 
 Now we show that $\rho(PSL(2,\C))=PSL(2,\C)$. Indeed, since for every rational maps $R$ 
 and $Q$  we have $\rho(R\ast_{R_1} 
 Q)=\rho(R)\ast_{R_2} \rho(Q)$,  taking $Q=c$ a constant then $f(R\circ R_1(c))=\rho(R)\circ 
 R_2(f(c)).$ As $f$ is 
 invertible we have $deg(R\circ R_1)=deg(\rho(R)\circ R_2).$ Similarly for $\rho^{-1}$ 
 we have $$f^{-1}(R\circ R_2(c))=\rho^{-1}(R)\circ R_1 (f^{-1}(c))$$ and 
 $$deg(R\circ R_2)=deg(\rho^{-1}(R)\circ R_1).$$
 If $R\in PSL(2,\C)$ then $deg(R_1)=deg(R\circ R_1)=deg(\rho(R))\cdot deg(R_2)=deg(\rho(R))\cdot 
 deg(\rho^{-1}(R))\cdot deg(R_1)$ which implies $deg(\rho(R))=1.$
 
 Let $\gamma=\rho(Id)$ so $\gamma\in PSL(2,\C)$. Consider the map 
 $\tau_\gamma:G_2\rightarrow \langle Rat, \ast_{\gamma \circ R_2}\rangle$
 given by $\tau_{\gamma}(R)= R\circ \gamma^{-1}.$ Then $\tau_\gamma$ is an 
 isomorphism of the sandwich semigroup, which follows from direct computation: 
 $$\tau_\gamma(R\ast_{R_2}Q)=\tau_\gamma(R\circ R_2 \circ Q)=R\circ R_2\circ Q \circ 
 \gamma^{-1}=\tau_\gamma(R) \ast_{\gamma\circ R_2} \tau_{\gamma}(Q).$$
 Then $\Phi:\tau_\gamma\circ \rho :G_1\rightarrow \langle Rat, \ast_{\gamma \circ 
 R_2}\rangle$ is an isomorphism satisfying $\Phi(Id)=Id.$ If $R=c$ is a constant then 
 $\Phi(c)=\rho(c)\circ \gamma=\rho(c)=f(c).$ Also 
 $\Phi(R_1)=\Phi(Id\ast_{R_1} Id)=\gamma \circ R_2$, and hence for every $c\in \C$ 
 $$f(R_1(c))=\Phi(Id\ast_{R_1}c)=\gamma\circ R_2(f(c))$$ 
 which implies $R_2=\gamma^{-1}\circ f \circ R_1\circ f^{-1}.$
 We have for every $Q\in Rat(\C)$ and $c\in \C$, 
 $$f(Q\circ R_1(c))=\Phi(Q\ast_{R_1} c)=\Phi(Q)\circ \gamma \circ R_2(f(c))=\Phi(Q)\circ f\circ R_1(c)$$
 then $$Q\circ R_1(c)=f^{-1}\circ \Phi(Q)\circ f \circ R_1(c)$$ and so
 $$\Phi(Q)=f\circ Q \circ f^{-1}$$ but $\Phi(Q)=\tau_\gamma \circ \rho (Q)$ or 
 $\rho(Q)=\Phi(Q)\circ \gamma$ as we wanted to prove. 
 \end{proof}

 As an immediate corollary we have the following. 
 \begin{corollary}\label{cor.sandId}
 If $\rho:G_1\rightarrow G_2$ is an isomorphism as in Theorem \ref{thIndbas} 
and  $\rho(Id)=Id$ then there exists a bijection $f:\overline{\C}\rightarrow 
\overline{\C}$ such that 
 $$\rho(R)=f\circ R \circ f^{-1}.$$
 \end{corollary}

 Next theorem is the main result of this section. 
 \begin{theorem}
  Let $R_1$ and $R_2$ be non constant rational maps and $G_1$ and $G_2$ are 
  the respective sandwich semigroups on $Rat(\C).$ Then 
  \begin{enumerate}
   \item The pair $R_1,R_2$ belongs to the same conformal Hurwitz class if and only if
   there exists a continuous orientation preserving isomorphism $\rho:G_1\rightarrow G_2$.
   Moreover,  $\rho(Id)=Id$ if and only if $R_1$ is $PSL(2,\C)$ conjugated to $R_2.$
   \item A continuous isomorphism $\rho$ reverses orientation if and only if $R_1$ is an 
   anticonformal copy of $R_2.$ Moreover $\rho(Id)=Id$ if and only if $R_1$ is anticonformally conjugated to $R_2.$  
  \end{enumerate}

 \end{theorem}

 \begin{proof}
 Part 1. Assume that $R_1$ and $R_2$ belong to the same conformal Hurwitz class and 
 let $h,g\in PSL(2,\C)$ be so that  $R_2=g^{-1}\circ R_1\circ h.$ Then the map
 $\rho(R)=h^{-1}\circ R \circ g$ defines a continuous orientation preserving 
 isomorphism from $G_1$ to $G_2.$ Indeed, from direct calculation:
 $$\rho(R\ast_{R_1} Q)=h^{-1}\circ R \circ R_1 \circ Q \circ g$$
 $$=\rho(R)\ast_{R_2} \rho(Q).$$
 Now if $\rho:G_1\rightarrow G_2$ is an orientation preserving isomorphism,
 then by  Theorem \ref{thIndbas} there exists a $\gamma\in PSL(2,\C)$ and 
a bijection $f:
 \overline{\C}\rightarrow \overline{\C}$ so that $\rho(Q)=f\circ Q\circ f^{-1}\circ \gamma$ for every
 rational map $Q.$ By conjugation,  $f$ defines an automorphism of $Rat(\C)$ with 
 composition as a product. By Proposition 8 in \cite{CaMakPl}, the map $f$ belongs to 
 the group generated by $PSL(2,\C)$ and the absolute Galois group. Since $\rho$
 is  continuous and orientation preserving we have $f\in PSL(2,\C)$. 
 Now assume $\rho(Id)=Id$, then by  Corollary \ref{cor.sandId}, $\rho(R)=f\circ R \circ f^{-1}$ and 
$f\in PSL(2,\C)$
 and $\rho(R_1)=f\circ R_1 \circ f^{-1}=\rho(Id\ast_{R_1} Id)=R_2$.
 
 Part 2. If $\rho$ is continuous and orientation reversing then $\overline{f}\in PSL(2,\C)$ and the proof
 goes as Part 1.
  \end{proof}
 In conclusion we note:

 First, it is possible to show that every continuous 
 semigroup product on   $Rat(\C)$ which is continuously isomorphic to a 
 sandwich product is  a sandwich product itself. So the classes of continuous 
 isomorphisms of sandwich
 semigroups on $Rat(C)$ correspond to the conformal Hurwitz classes of rational maps.
 
 Second that it is not clear at all how to associate the algebraic characterization of the 
 conformal  Hurwitz class of symmetric rational maps with the geometric cobordisms point of 
 view.
 
 Finally, similar ideas allow us to consider the Hurwitz space as a representation space of
 a special class of holomorphic correspondences. This follows using results from 
 \cite{CaMakPl} with \cite{CabMakSie}. So we can construct a Teichm\"uller space
 of correspondences of the form $R^{-1}\circ R$, called the \textit{deck correspondence 
 associated to} $R.$ From \cite{CabMakSie} it follows that the Speisser class of $R$
 fibers over the Moduli space of the deck with fiber equivalent to the conformal 
 Hurwitz class of $R.$ Let $G_R=\langle R^{-1}\circ R,\overline{\C}\rangle$ be the semigroup of 
 holomorphic correspondences generated by the deck correspondence associated to $R$ 
 and constant maps. Consider the space $\mathcal{X}$  of all representations of $G_R$
 into the semigroup of holomorphic correspondences on $\overline{\C}.$
 Then using results from \cite{CaMakPl}, one can  consider 
 the Speisser class of a rational map $R$ as a subspace of the connected component of $\mathcal{X}$ containing the identity representation.

\bibliographystyle{amsplain} 
\bibliography{workbib}

\end{document}